\documentclass[final]{siamltex1213}
\usepackage{amssymb}
\usepackage{mathrsfs}

\newcommand{\dbR}{\mathbb{R}}

\newtheorem{remark}{\textit{Remark}}[section]

\newtheorem{mydefinition}{Definition}[section]

\makeatletter

\def\leqalignno#1{\displ@y \tabskip \@centering \halign to\displaywidth
{\hfil  $\@lign \displaystyle {##}$\tabskip \z@skip &$\@lign
\displaystyle  {{}##}$\hfil \tabskip \@centering &\kern -\displaywidth
\rlap  {$\@lign ##$}\tabskip \displaywidth \crcr #1\crcr }}

\def\leqalign#1{\null \, \vcenter
{\openup \jot \m@th \ialign {\strut \hfil $\displaystyle {##}
$&$\displaystyle {{}##}$\hfil \crcr #1\crcr }}\,}

\def\eqalign#1{\null \,\vcenter
{\openup \jot \m@th \ialign {\strut \hfil $\displaystyle {##}
$&$\displaystyle {{}##}$\hfil \crcr #1\crcr }}\,}

\def\eqalignno#1{\displ@y \tabskip\@centering \halign to\displaywidth
{\hfil$\@lign\displaystyle{##}$\tabskip\z@skip &
$\@lign\displaystyle{{}##}$\hfil\tabskip\@centering
    &\llap{$\@lign##$}\tabskip\z@skip\crcr
    #1\crcr}}
\makeatother

\def\sqr#1#2{{\vcenter{\vbox{\hrule height.#2pt
              \hbox{\vrule width.#2pt height#1pt \kern#1pt \vrule width.#2pt}
          \hrule height.#2pt}}}}
\def\a{\alpha}
\def\b{\beta}

\def\d{\delta}

\def\k{\kappa}
\def\l{\lambda}

\def\O{\Omega}
\def\eq{\eqalign}

\def\ms{\medskip}

\def\essinf{\mathop{\rm ess.inf}}

\def\lds{\ldots}

\title{Dynamic Programming for General  Linear Quadratic Optimal Stochastic Control with Random
Coefficients
\thanks{
This research was supported by  the National Natural Science Foundation of China (Grants \#10325101 and \#11171076), and by Science and Technology Committee, Shanghai Municipality (Grant No. 14XD1400400).
}}

\author{Shanjian Tang\thanks{Department of Finance and Control Sciences, School of Mathematical Science, Fudan University,
Shanghai 200433, China (sjtang@fudan.edu.cn).}}

\begin{document}

\maketitle

\setcounter{page}{1}

\pagestyle{myheadings}
\markboth{SHANJIAN TANG}{DYNAMIC PROGRAMMING FOR STOCHASTIC LQ}

\begin{abstract}
We are concerned with the linear-quadratic optimal stochastic control problem where all the coefficients of the control system and the running weighting matrices in the cost functional are allowed to be predictable (but essentially bounded) processes and the terminal state-weighting matrix in the cost functional is allowed to be random. Under suitable conditions,
we prove that the value field $V(t,x,\omega), (t,x,\omega)\in [0,T]\times R^n\times \Omega$, is quadratic in $x$, and has the following form: $V(t,x)=\langle K_tx, x\rangle$ where $K$ is an essentially bounded nonnegative symmetric matrix-valued  adapted processes. Using the dynamic programming principle (DPP), we prove that $K$ is a continuous semi-martingale of the form $$K_t=K_0+\int_0^t \, dk_s+\sum_{i=1}^d\int_0^tL_s^i\, dW_s^i, \quad t\in [0,T]$$
with $k$ being a continuous process of bounded variation and
$$E\left[\left(\int_0^T|L_s|^2\, ds\right)^p\right] <\infty, \quad \forall p\ge 2; $$
and that $(K, L)$ with $L:=(L^1, \cdots, L^d)$ is a solution to the associated backward stochastic Riccati equation (BSRE), whose generator is highly nonlinear in the unknown pair of processes. The uniqueness is also proved via a localized completion of squares in a self-contained manner for a general BSRE.  The existence and uniqueness of adapted solution to a general BSRE was initially proposed by the French mathematician J. M. Bismut
[in \textit{SIAM J. Control \& Optim.}, 14(1976), pp.\ 419--444, and in \textit{S\'eminaire de Probabilit\'es} XII, Lecture Notes in Math. 649,
C. Dellacherie, P. A. Meyer, and M. Weil, eds., Springer-Verlag, Berlin, 1978, pp.\ 180--264], and subsequently listed by Peng [in \textit{Control of Distributed Parameter and Stochastic Systems (Hangzhou, 1998)},
S. Chen, et al., eds., Kluwer Academic Publishers, Boston, 1999, pp.\ 265--273] as the first open problem for backward stochastic differential equations. It had remained to be open until a general solution by the author [in \textit{SIAM J. Control \& Optim.}, 42(2003), pp.\ 53--75] via the stochastic maximum principle with a viewpoint of stochastic flow for the associated stochastic Hamiltonian system. The present paper is its companion, and gives the {\it second but more comprehensive} (seemingly much simpler, but appealing to the advanced tool of Doob-Meyer decomposition theorem, in addition to the DDP) adapted solution to a general BSRE via the DDP. Further extensions to the jump-diffusion control system and to the general nonlinear control system are possible.
\end{abstract}



\section{Formulation of the problem and basic assumptions}\label{sec1}
Consider the following linear quadratic  optimal stochastic control (SLQ in short
form) problem: minimize over $u\in
\mathscr{L}^2_\mathscr{F}(0,T; \dbR^m)$ the following quadratic cost functional:
\begin{equation}\label{eq1.1}
J(u;0,x):=E^{0,x;u}\left[\langle MX_T,X_T\rangle +\int_0^T(\langle
Q_sX_s,X_s\rangle +\langle N_su_s,u_s\rangle )\, ds\right],
\end{equation}
where $X$ is the solution of the following linear stochastic
control system:
\begin{equation}\label{eq1.2}
\left\{\eq{dX_t=&\;(A_tX_t+B_tu_t)\, dt
+\sum_{i=1}^d(C_t^iX_t+D_t^iu_t)\, dW_t^i,\cr X_0=&\;x\in \dbR^n.
\cr}\right.
\end{equation}
Here, $\{W_t:=(W_t^1, \ldots, W_t^d)', 0\le t\le T\}$ is a
$d$-dimensional standard Brownian motion defined on some
probability space $(\O, \mathscr{F}, P)$. Denote by $\{\mathscr{F}_t,0\le t\le T\}$ the augmented natural filtration of the
standard Brownian motion $W$. The control $u$ belongs to the
Banach space $\mathscr{L}^2_\mathscr{F}(0,T;
\dbR^m)$, which consists of all $\dbR^m$-valued square integrable
$\{\mathscr{F}_t, 0\le t\le T\}$-adapted processes.  Denote by $\mathbb{S}^n$ the totality of $n\times n$ symmetric matrices, and by $\mathbb{S}^n_+$ the totality of $n\times n$ nonnegative matrices.

Throughout this paper, we make the following two assumptions on the coefficients of
the above problem.

(A1) Assume that the matrix processes $A: [0,T]\times \O\to
\dbR^{n\times n}$, $B:[0,T]\times \O\to \dbR^{n\times m}$;
$C^i:[0,T]\times \O\to \dbR^{n\times n}$, $D^i:[0,T]\times \O\to
\dbR^{n\times m}$, $i=1,\lds,d$;
$Q:[0,T]\times \O\to \mathbb{S}^n_+$, $N:[0,T]\times \O\to \mathbb{S}^m_+ $
 and the random matrix $M:\O\to \mathbb{S}^n_+$ are uniformly bounded and
 $\{{\cal F }_t, 0\le t\le T\}$-adapted or ${\cal F}_T$-measurable.


  (A2)  Assume that  the control weighting matrix process $N$ is  uniformly
 positive.

 \vskip0.5cm
 Define for $(t,K, L)\in [0,T]\times \mathbb{S}^n\times (\mathbb{S}^n)^d$,
\begin{equation}
  \eq{\mathscr{N}_t(K):=& N_t+\sum_{i=1}^d(D^i_t)'KD^i_t, \cr
   \mathscr{M}_t(K,L):=& KB_t+\sum_{i=1}^d(C^i_t)'KD^i_t+\sum_{i=1}^dL^iD^i_t. \cr}
\end{equation}
 For  $(t, K)\in [0,T]\times \mathbb{S}^n_+$ and $L=(L^1,\ldots, L^d)\in (\mathbb{S}^n)^d$, define
\begin{equation}\label{G}
\eq{G(t, K,L):=&\  A'_tK+KA_t+Q_t+\sum_{i=1}^d(C^i_t)'KC^i_t+\sum_{i=1}^d[(C^i_t)'L^i+L^iC^i_t]\cr
&\quad\quad -\mathscr{M}_t(K,L)\mathscr{N}_t^{-1}(K)\mathscr{M}_t'(K,L).\cr}
\end{equation}
Here, we use the prime  to denote the transpose of a vector or a matrix.
Associated to the above SLQ problem is the following backward stochastic Riccati equation (BSRE):
\begin{equation}\label{bsre}
\left\{\eq{dK_t=&\;-G(t, K_t,L_t)\, dt+\sum_{i=1}^dL_t^i\,
dW_t^i, \quad t\in [0,T);\cr K_T=&\;M,\qquad L_t:=(L_t^1, \ldots, L_t^d).\cr}\right.
\end{equation}
The generator is {\it highly nonlinear} in the  unknown pair of variables $(K, L)$.

\begin{mydefinition}\label{solution bsre}
A solution of BSRDE (\ref{bsre}) is defined as a pair $(K,L)$
of  matrix-valued adapted processes such that

{\rm (i)} $\int_0^T|L_t|^2\, dt + \int_0^T|G(t,K_t,L_t)|\, dt<\infty,\ a.s.;$

{\rm (ii)} The $m\times m$ matrix-valued process $\{\mathscr{N}_t(K_t), t\in [0,T]\}$ is $a.s.a.e.$
positive; and

{\rm (iii)} $K_t=M+\int_t^TG(s, K_s, L_s)\, ds-\int_t^T\sum_{i=1}^dL^i_s\, dW^i_s $ a.s. for all
$t\in [0,T].$
\end{mydefinition}

\ms
The adapted solution to a general BSRE~(\ref{bsre}) was initially proposed by the French mathematician J. M. Bismut~\cite{Bismut1976,Bismut1978}, and subsequently listed by Peng~\cite{Peng1999} as the first open problem for backward stochastic differential equations. It had remained to be open until a general solution by the author~\cite{Tang2003} via the stochastic maximum principle and using a viewpoint of stochastic flow for the associated stochastic Hamiltonian system. For more details on the historical studies on BSRE~(\ref{bsre}) and the progress, see the author's previous paper~\cite[Section 4, pages 60--61]{Tang2003} and the plenary lecture by Peng~\cite{Peng2010} at the International Congress of Mathematicians in 2010. In the paper, we shall give a novel proof to the existence for BSRE~(\ref{bsre}) via dynamic programming principle. A crucial point is that we can show the value field is a semi-martingale of both ``sufficiently good" parts of bounded variation and martingale.

The rest of our paper is organized as follows. Section 2 gives preliminaries. In Section 3, we prove that the value field $V(t,x, \omega)$ is quadratic in $x$.
In Section 4, we prove that the value field is a semi-martingale and  that BSRE~(\ref{bsre}) has an adapted solution. Section 5 is concerned with a verification theorem for the SLQ problem, and the uniqueness of solution to BSRE~(\ref{bsre}). Finally, in Section 6, we give some comments and possible extensions.

\section{Preliminaries}\label{sec2}

For each $u\in \mathscr{L}^2_\mathscr{F}(0,T; \dbR^m)$, the following linear stochastic differential equation
\begin{equation}\label{eq2.2}
\left\{\eq{dX_t=&(A_tX_t+B_tu_t)\, dt
+\sum_{i=1}^d(C_t^iX_t+D_t^iu_t)\, dW_t^i,\qquad \tau \le t\le T,
\cr X_s=& x\in \mathbb{R}^n,
\cr}\right.
\end{equation}
has a unique strong solution (see Bismut~\cite{Bismut1978}), denoted by $X^{s,x;u}$ with the superscripts indicating the dependence on the initial data $(s,x)$ and the control action.
We have the following well-known quantitative dependence of the solution $X^{s,x;u}$ on the initial data $(s,x)$ and the control action $u$.

\begin{lemma}\label{sde estimate} Let assumption (A1) be satisfied. For any $p\ge 1$, there is a positive constant $C_p$ such that for any initial state $\xi\in L^p(\Omega, \mathscr{F}_s, P;\mathbb{R}^n)$ and predictable control $u$ with
$$
E\left[\left(\int_s^T|u_r|^2\, dr\right)^{p/2}\right]<\infty,
$$
we have
\begin{equation}\label{state-estimate}
E\left[\max_{t\in [s,T]}|X^{s,\xi;u}|^p\biggm | \mathscr{F}_s\right]\le C_p\left(|\xi|^p+E\left[\left(\int_s^T|u_r|^2\, dr\right)^{p/2}\biggm| \mathscr{F}_s\right]\right).
\end{equation}
\end{lemma}

Consider  the initial-data-parameterized
SLQ problem: minimize over $u\in \mathscr{L}^2_\mathscr{F}(0,T; \dbR^m)$ the quadratic cost functional
\begin{equation}\label{eq2.1}
\hspace*{12pt}J(u; s,x):=E^{s,x;u}\left[\langle MX_T, X_T\rangle
+\int_s^T(\langle Q_rX_r, X_r\rangle
+\langle N_ru_r,u_r\rangle )\, dr\biggm | \mathscr{F}_s\right].
\end{equation}
Define the value field
\begin{equation}\label{value}
V(s,x):=\essinf_{u\in \mathscr{L}^2_{\mathscr{F}_t}(s,T; \mathbb{R}^m)}J(u;s,x), \quad (s,x)\in [0,T]\times \mathbb{R}^n.
\end{equation}

Assumptions (A1) and (A2) imply that the above
SLQ problem has a unique optimal control for any $\xi\in L^2(\Omega, \mathscr{F}_s, P;\mathbb{R}^n)$, that is,  there is unique ${\overline u}\in \mathscr{U}_s$ such that
$$
V(s, \xi)=J({\overline u}; s, \xi).
$$
See Bismut~\cite{Bismut1978} for the
proof of such a result. A further step is to characterize the
optimal control.

We easily prove the following

\begin{lemma}\label{bound} Let Assumptions (A1) and (A2) be satisfied.  There is a positive constant $\lambda$  such that
$$
0\le V(s,\xi)\le J(0;s,\xi)\le \lambda |\xi|^2, \quad \forall (s,\xi) \in [0,T]\times L^2(\Omega, \mathscr{F}_s, P;\mathbb{R}^n).
$$
\end{lemma}

\begin{proof} In view of assumption (A1) and the definition of the value field $V$,  it is sufficient to show $J(0;s,\xi)\le \lambda |\xi|^2$, which is an immediate consequence of Lemma~\ref{state-estimate} and the following estimate:
$$\eq{J(0;s,\xi)\le\ & \lambda E\left[|X_T^{0,\xi;0}|^2+\int_0^T|X_t^{0,\xi;0}|^2\, dt\ \biggm |\mathscr{F}_s\right]\cr
\le\  & \lambda (1+T) E\left[\max_{t\in [0,T]}|X_t^{0,\xi;0}|^2\biggm |\mathscr{F}_s\right].\cr}
$$
\end{proof}

\section{The value field $V$ is quadratic in the space variable}
This section is an adaptation of Faurre~\cite{Faurre1968} to our SLQ problem with random coefficients.

We have
\begin{theorem} \label{quadratic} Let Assumptions (A1) and (A2) be satisfied.  The value field $V(s,x)$ is quadratic in $x$. Moreover, there is an essentially bounded  continuous nonnegative matrix-valued process $K$ such that
\begin{equation}
    V(s,x)=\langle K_s x,x \rangle, \quad \forall (s,x)\in [0,T]\times \mathbb{R}^n.
\end{equation}
\end{theorem}

The state-quadratic property follows from  the following lemma.

\begin{lemma} Let Assumptions (A1) and (A2) be satisfied. The value field has the following two laws  in the state variable $x$ of (i) square homogeneity
$$
V(s, \xi x)=\xi^2 V(s,x), \quad \forall (s, x, \xi) \in [0,T]\times \mathbb{R}^n\times L^\infty(\Omega, \mathscr{F}_s, P)
$$
and (ii) parallelogram
$$
V(s, x+y)+V(s,x-y)=2 V(s,x)+2V(s,y), \quad \forall (s, x, y) \in [0,T]\times \mathbb{R}^n\times \mathbb{R}^n.
$$
\end{lemma}

\begin{proof} It is easy to derive from the linearity of the control system and the quadratic structure of the cost functional  the following two identities for any $u\in \mathscr{U}_s$,
$$
\xi X^{s,x;u}=X^{s,\xi x; \xi u}  , \quad \xi^2 J(u;s,  x)= J(\xi u; s, \xi x).
$$
Therefore, we have
$$
\xi^2 V(s,x)=\xi^2 \essinf_{u\in \mathscr{U}_s} J(u;s,x)=\essinf_{u\in \mathscr{U}_s} \xi^2 J(u;s,x)=\essinf_{u\in \mathscr{U}_s} J(\xi u;s,\xi x),
$$
which is equal to $V(s, \xi x)$ by definition,  immediately giving assertion (i).

Let us show assertion (ii). It is easy to see (see Bismut~\cite{Bismut1978})
that there are $\a, \b \in \mathscr{U}_s$ such that
$$
V(s,x+y)=J(\a; s,x+y), \quad V(s,x-y)=J(\b;s,x-y).
$$
Then, we easily see that
$$
V(s, (x+y)\pm (x-y))\le J(\a\pm \b; s, (x+y)\pm (x-y))
$$
and therefore,
$$
V(s, 2x)+V(s, 2y)\le J(\a + \b; s, 2x)+J(\a- \b; s, 2y).
$$
Since $J(u;s,x)$ is quadratic in the pair $(u,x)$ and satisfies the parallelogram
$$
2J(\a + \b; s, 2x)+2J(\a- \b; s, 2y)= J(2\a; s,2(x+y))+J(2\b;s,2(x-y)),
$$
we have
$$
V(s, 2x)+V(s, 2y)\le {1\over2}[ J(2\a; s,2(x+y))+J(2\b;s,2(x-y))],
$$
and therefore by the square homogeneity of $J(u;s,x)$ in the pair $(u,x)$
$$
V(s, x+y)+V(s,x-y)\le 2J(\a; s, x+y)+2J(\b;s,x-y)=2 V(s,x)+2V(s,y).
$$
By symmetry, it holds for $x':=x+y$ and $y':=x-y$:
$$
V(s,(x+y)+(x-y))+V(s,(x+y)-(x-y))\le 2 V(s, x+y)+2V(s, x-y)
$$
which leads by assertion (i) to the following desired reverse inequality
$$
4V(s,x)+4V(s,y)=V(s,2x)+V(s,2y)\le 2V(s,x+y)+2V(s,x-y).
$$
The proof is then complete.
\end{proof}

The nonnegativity and the essential bound of the process $K$ are immediate consequences of Lemma~\ref{bound}.

\section{Dynamic programming principle and the semi-martingale property of the value field}\label{sec3} For simplicity, define the function
\begin{equation}
l(t,x,u):=\langle Q_tx, x\rangle
+\langle N_tu,u\rangle, \quad (t,x,u)\in [0,T]\times \mathbb{R}^n\times \mathbb{R}^m
\end{equation}
and the set
\begin{equation}
\mathscr{U}_s:=\mathscr{L}^2_{\mathscr{F}}(s,T; \mathbb{R}^m).
\end{equation}
We denote by $\mathbb{V}(t, \cdot)$ the restriction of $V(t, \cdot)$ to $\mathbb{R}^n$. By definition, we have almost surely
$$
V(t,x)=\mathbb{V}(t,x), \quad \forall \  x \in \mathbb{R}^n.
$$
For any $\xi\in L^2(\Omega, \mathscr{F}_t, P; \mathbb{R}^n)$, in an analogous way to the  proof of Peng~\cite[Lemma 6.5, page 122]{Peng1997}, we also have  almost surely
$$
V(t,\xi)=\mathbb{V}(t,\xi).
$$

We have

\begin{theorem} \label{BP}(Bellman's Principle). Let Assumptions (A1) and (A2) be satisfied. We have

(i) For $s\le t\le T$ and $\xi\in L^2(\Omega, \mathscr{F}_s, P; \mathbb{R}^n)$,
$$
\mathbb{V}(s,\xi)=\mbox{\rm ess.}\inf_{u\in {\mathscr U}_s} E^{s,\xi; u} \left \{\int_s^tl(r, X_r, u_r)\, dr+ \mathbb{V}(t, X_t)  \biggm | {\mathscr F}_s\right\}.
$$
For the optimal control $\overline{u}\in \mathscr{U}_s$, we have
$$
\mathbb{V}(s,\xi)= E^{s,\xi; \overline{u}} \left \{\int_s^tl(r, X_r, \overline{u}_r)\, dr+ \mathbb{V}(t, X_t)  \biggm | {\mathscr F}_s\right\}.
$$

(ii) For $(s,x,u)\in [0,T]\times \mathbb{R}^n\times {\mathscr U}_s$, the process
$$\begin{array}{rcl}
 \k_t^{s,x;u}&:=&\displaystyle  \mathbb{V}(t,X_t^{s,x;u}) +\int_s^tl(r, X_r^{s,x;u}, u_r)\, dr
\end{array}
$$
defined for $t\in [s,T]$, is a submartingale w.r.t. $\{{\mathscr F}_t\}$; and for the optimal control $\overline{u}\in \mathscr{U}_s$, the process  $\k_t^{s,x;\overline{u}}, t\in [s,T]$, is a martingale w.r.t. $\{{\mathscr F}_t\}$.
\end{theorem}

\begin{proof} It is easy to check that Assertion (ii) is an immediate consequence of Assertion (i). Assertion (i) is more or less standard, and the proof  is similar to that of Krylov~\cite[Theorem 6, Section 3, Chapter 3, page 150]{Krylov1977} or Peng~\cite[Theorem 6.6, page 123]{Peng1997}.
\end{proof}

From assertion (i), we have

\begin{corollary}\label{con mean cont} We have the following time continuity of $\mathbb{V}$ and $K$:  for any $(s,x)\in  [0,T]\times \mathbb{R}^n$, 
$$
\lim_{t\to s}E[\mathbb{V}(t,x)-\mathbb{V}(s,x)\, |\mathscr{F}_s]=\ 0, \quad
 \lim_{t\to s} E[K_t-K_s\, |\mathscr{F}_s]=\  0, \quad a.s..
$$
\end{corollary}

\begin{proof} In view of Theorem~\ref{quadratic}, the second limit easily follows from the first one. It remains to prove the first limit. 

Assume without loss of generality that $s\le t$. We have 
$$
\mathbb{V}(s,x)= E^{s,x; \overline{u}} \left \{\int_s^tl(r, X_r, \overline{u}_r)\, dr+ \mathbb{V}(t, X_t)  \biggm | {\mathscr F}_s\right\}
$$
where $\overline{u}\in \mathscr{U}_s$ is the optimal control. Therefore, 
$$
|E[\mathbb{V}(t,x)-\mathbb{V}(s,x)\, |\mathscr{F}_s]|\le E^{s,x; \overline{u}} \left \{\int_s^tl(r, X_r, \overline{u}_r)\, dr+ |\mathbb{V}(t, X_t)-\mathbb{V}(t,x)|  \biggm | {\mathscr F}_s\right\}. 
$$
Since
$$
|\mathbb{V}(t, X_t^{s,x; \overline{u}})-\mathbb{V}(t,x)|\le \l (|x|+|X_t^{s,x; \overline{u}}|) |X_t^{s,x; \overline{u}}-x|, 
$$
using estimate~(\ref{state-estimate}), we have 
$$
\eq{&|E[\mathbb{V}(t,x)-\mathbb{V}(s,x)\, |\mathscr{F}_s]|\le \ \l E^{s,x; \overline{u}} \left \{\int_s^t ( |X_r|^2+|\overline{u}_r|^2)\, dr \biggm | {\mathscr F}_s\right\}\cr
&\ +\l \left\{|x|+E^{s,x; \overline{u}}\left[ \left(\int_s^t|\overline{u}_r|^2\, dr\right)^{1/2} \biggm | {\mathscr F}_s\right]\right\} E^{s,x; \overline{u}}\left[ \left(\int_s^t|\overline{u}_r|^2\, dr\right)^{1/2} \biggm | {\mathscr F}_s\right], \cr}
$$
which implies the desired limit. 
\end{proof}

Using Theorems~\ref{quadratic} and ~\ref{BP} , we can prove the following

\begin{theorem} \label{rep}
The value field $V$ is a semi-martingale of the following representation:
\begin{equation}\label{quadratic form}
\mathbb{V}(t,x)=\langle K_tx, x\rangle
\end{equation} where $K$ is an essentially bounded nonnegative symmetric matrix-valued continuous semi-martingale of the form
\begin{equation}\label{representation} K_t=K_0-\int_0^t d k_s +\sum_{i=1}^d\int_0^tL_s^i\, dW_s^i, \quad t\in [0,T]; \quad K_T=M\end{equation}
with $k$ being an $n\times n$ atrix-valued continuous process of bounded variation such that
\begin{equation}\label{formula}
\eq{ dk_s= &\  G(s,K_s,L_s)\, ds, \quad \hbox{ \rm almost everywhere } (s,\omega) \in [0,T]\times \Omega. \cr}
\end{equation}
and
\begin{equation}
\label{p-estimate}
E\left[\left(\int_0^T|L_s|^2\, ds\right)^p\right] <\infty, \quad \forall p\ge 2. \end{equation}
\end{theorem}

\begin{proof} Theorem~\ref{quadratic} states that there is an essentially bounded nonnegative symmetric matrix-valued process $K$ such that (\ref{quadratic form}) holds true. The rest of the proof is divided into the following three steps.

{\bf Step 1.  $K$ is a semi-martingale of form~(\ref{representation}) in the Doob-Meyer decomposition. }
Let $e_i$ be the unit column vector of $\mathbb{R}^n$ whose $i$-th component is the number $1$ for $i=1,\ldots,n$. In view of Assertion (ii) of Theorem~\ref{BP},  we see that  for $x=e_i, e_i+e_j, e_i-e_j, i,j=1,\ldots,n$, 
$\{\k_t^{0,x;0}, t\in [0,T]\}$ is a sub-martingale,  and since 
$$|\k_t^{0,x;0}|\le \l |X_t^{0,x;0}|^2+\int_0^t |X_s^{0,x;0}|^2\, ds\le \l \max_{t\in [0,T]}|X_t^{0,x;0}|^2 \in L^1(\Omega, \mathscr{F}_T, P),$$ 
it is of class $D$. Since $V(t,x)$ is continuous in the sense of conditional mean in $t$ (see corollary~\ref{con mean cont}), $\{\k_t^{0,x;0}, t\in [0,T]\}$  is continuous in the sense of conditional mean in $s$.  In view of Doob-Meyer decomposition (see Protter~\cite[Theorem 11, page 112]{Protter2005}),  its bounded variational process  is continuous and increasing in time,  and $\{\k_t^{0,x;0}, t\in [0,T]\}$ is sample continuous.
Define the $n\times n$ symmetric matrix-valued process
\begin{equation}\label{submartingales}
    \Gamma_t:=(\k_t(i,j))_{1\le i, j\le n}
\end{equation}
where
\begin{equation}\label{kappa}
    \k_t(i,i):=\k_t^{0,e_i;0}, \quad \k_t(i,j):={1\over 4} [\k_t^{0,e_i+e_j;0}-\k_t^{0,e_i-e_j;0}], \quad 1\le i\not=j\le n.
\end{equation}
It is a $n\times n$ matrix-valued semi-martingale and the bounded variational process  in the Doob-Meyer decomposition is continuous in time.
Define
$$
\Phi_t:=(X_t^{0,e_1;0}, \cdots, X_t^{0,e_n;0}), \quad t\in [0,T].
$$
Then, we have
\begin{equation}\label{submartingales1}
    \Gamma_t=\Phi_t'K_t\Phi_t+\int_0^t\Phi_r'Q_r\Phi_r\, dr, \quad t\in [0,T];
\end{equation}
 and $\Phi$ satisfies the following matrix-valued stochastic differential equation (SDE):
\begin{equation}\label{SDE}
d\Phi_t= A_t \Phi_t\, dt+C_t^i\Phi_t\, dW_t^i, \quad t\in (0,T]; \quad \Phi_0= I_n.
\end{equation}
It is well-known that $\Phi_t$ has an inverse $\Psi_t:=\Phi_t^{-1}$, satisfying the following SDE:
\begin{equation}\label{SDE1}
d\Psi_t= \Psi_t (-A_t+C_t^iC_t^i)\, dt-\Psi_t C_t^i\, dW_t^i, \quad t\in (0,T]; \quad \Psi_0= I_n.
\end{equation}
Therefore, we have
\begin{equation}\label{semimartingales1}
  K_t= \Psi_t'\left( \Gamma_t  -\int_0^t\Phi_r'Q_r\Phi_r\, dr \right)\Psi_t, \quad t\in [0,T].
\end{equation}
Since $\Gamma$ is a semi-martingale, using It\^o-Wentzell formula, we see that  $K$ is a semi-martingale of form~(\ref{representation}) from the Doob-Meyer decomposition, with the bounded variational process $k$ being continuous in time. It remains to derive the formula~(\ref{formula}) for $k$ and  the estimate~(\ref{p-estimate}) for $L$.

{\bf Step 2. Formula for the bounded variational process $k$. } Define the function:
\begin{equation}\label{Hamiltonian}
  \eq{ F(t, x,v; K, L)=& 2\langle Kx, A_t x+B_tv\rangle + 2\langle L^ix, C_t^ix+D_t^iv\rangle\cr
  & +\langle L^i (C_t^ix+D_t^iv), C_t^ix+D_t^iv \rangle, \cr}
\end{equation}
for $(t,x,v, K,L)\in [0,T]\times \mathbb{R}^n\times \mathbb{R}^m\times \mathbb{S}^n\times (\mathbb{S}^n)^m$.
Using It\^o-Wentzell formula, we have
\begin{equation}\label{submartingales2}\left\{
  \eq{  dV(t, X_t^{0,x;v})=&\biggl [-\langle d k_t X_t^{0,x;v}, X_t^{0,x;v}\rangle+F(t,X_t^{0,x;v}, v; K_t, L_t)\, dt\biggr]\cr
  &+\biggl[\langle K_t (C_t^iX_t^{0,x;v}+D_t^iv), X_t^{0,x;v}\rangle \cr
  &+ \langle K_t X_t^{0,x;v}, (C_t^iX_t^{0,x;v}+D_t^iv)\rangle\cr
  &\quad\quad +\langle L_t^iX_t^{0,x;v}, X_t^{0,x;v}\rangle\biggr]\, dW_t^i,\quad t\in [0,T); \cr
   V(T,X_T^{0,x;v})=&\langle M X_T^{0,x;v}, X_T^{0,x;v}\rangle. \cr}\right.
\end{equation}
and
\begin{equation}\label{submartingales3}
  \eq{  \k_t^{0,x;v}=&\langle K_0 x, x\rangle+ \int_0^t\biggl [-\langle d k_s X_s^{0,x;v}, X_s^{0,x;v}\rangle+F(s,X_s^{0,x;v}, v; K_s, L_s)\, ds\cr
  &+l(s, X_s^{0,x;v}, v)\, ds\biggr]+\int_0^t\biggl[\langle K_s (C_s^iX_s^{0,x;v}+D_s^iv), X_s^{0,x;v}\rangle\cr
  & + \langle K_s X_s^{0,x;v}, (C_s^iX_s^{0,x;v}+D_s^iv)\rangle +\langle L_s^iX_s^{0,x;v}, X_s^{0,x;v}\rangle\biggr]\, dW_s^i, \quad t\in [0,T].\cr}
\end{equation}
Assertion (ii) of Theorem~\ref{BP} states that $ \{\k_t^{0,x;v}, t\in [0,T]\}$ is a sub-martingale for any $(v,x)\in \mathbb{R}^m\times \mathbb{R}^n$, yielding the following fact: for any $(x,v)\in \mathbb{R}^n\times \mathbb{R}^m$, we have $E\int_0^T\eta (s,x) \gamma(ds,x; v)\le 0$ for any essentially bounded nonnegative predictable process $\eta$ on $[0,T]\times \Omega$, where
\begin{equation}\label{submartingales4}
 \eq{ \gamma(ds,x;v):=& -\langle d k_s X_s^{0,x;v}, X_s^{0,x;v}\rangle+F(s,X_s^{0,x;v}, v; K_s, L_s)\, ds\cr
 & \quad \quad +l(s, X_s^{0,x;v}, v)\, ds; \cr}
\end{equation}
and for the optimal control $\overline{u}\in \mathscr{U}_0$, the process  $\k_t^{s,x;\overline{u}}, t\in [s,T]$, is a martingale w.r.t. $\{{\mathscr F}_t\}$, yielding the following fact:
for any $x\in \mathbb{R}^n$, we have  $E\int_0^T\eta (s,x) \gamma(ds,x;\overline{u})= 0$ for any essentially bounded nonnegative predictable process $\eta$ on $[0,T]\times \Omega$, where
\begin{equation}\label{martingales4}
 \eq{ \gamma(ds,x;\overline{u}):=& -\langle d k_s X_s^{0,x;\overline{u}}, X_s^{0,x;\overline{u}}\rangle+F(s,X_s^{0,x;\overline{u}}, \overline{u}_s; K_s, L_s)\, ds\cr
 & \quad \quad +l(s, X_s^{0,x;\overline{u}}, \overline{u}_s)\, ds. \cr}
\end{equation}

It is well-known that the stochastic flow  $X_s^{0,x;v}, x\in \mathbb{R}^n$ has an inverse $Y_s^{0,x;v}, x\in \mathbb{R}^n$. Since (see Yong and Zhou~\cite[Theorem 6.14, page 47]{YongZhou})
\begin{equation}\label{sde flow}
    X_s^{0,x;v}=\Phi_t x+\Phi_t\int_0^t\Psi_s(B_sv-C_s^iD_s^iv)\, ds+\Phi_t\int_0^t\Psi_s D_s^i v\, dW_s^i
\end{equation}
for $t\in [0,T]$, we have
\begin{equation}\label{inverse flow}
    Y_s^{0,x;v}=\Psi_t x-\int_0^t\Psi_s(B_sv-C_s^iD_s^iv)\, ds-\int_0^t\Psi_s D_s^i v\, dW_s^i, \quad t\in [0,T].
\end{equation}
More generally, we define for any $u\in \mathscr{U}_0$ and $t\in [0,T]$,
\begin{equation}\label{g-inverse flow}
    Y_s^{0,x;u}=\Psi_t x-\int_0^t\Psi_s(B_su_s-C_s^iD_s^iu_s)\, ds-\int_0^t\Psi_s D_s^i u_s\, dW_s^i.
\end{equation}
We have
\begin{equation}
   X_s^{0,y;u} \biggm|_{y=Y_s^{0,x;u}}=x, \quad \forall x\in \mathbb{R}^n.
\end{equation}

Incorporating the composition of $\gamma(s,\cdot;v)$ with the inverse flow $Y_s^{0,x;v}, x\in \mathbb{R}^n$, we have
\begin{equation}\label{submartingales5}
\eq{ 0\le &\ \gamma(ds,Y_s^{0,x;v};v)\cr
=& -\langle d k_s x, x\rangle+\left[F(s,x, v; K_s, L_s) +l(s, x, v)\right]\, ds \cr}
\end{equation}
and in a similar way, we have for almost everywhere $(s,\omega) \in [0,T]\times \Omega$,
\begin{equation}\label{martingales5}
\eq{ 0=&\ \gamma(ds,Y_s^{0,x;\overline{u}};\overline{u})\cr
=& -\langle d k_s x, x\rangle+\left[F(s,x, \overline{u}_s; K_s, L_s) +l(s, x, \overline{u}_s)\right]\, ds.\cr}
\end{equation}
Therefore, we have
\begin{equation}\label{drift}
 \langle d k_s x, x\rangle= \min_{v\in \mathbb{R}^m} \left[F(s,x,v; K_s, L_s) +l(s, x, v)\right]\, ds, \quad \forall x\in \mathbb{R}^n,
\end{equation}
which implies formula~(\ref{formula}).

{\bf Step 3. Estimate for $L$. }

From the theory of BSDEs, we have from BSDE~(\ref{submartingales2})
\begin{equation}\label{L-estimate0}
  \eq{ &\int_0^T\left|\langle K_t X_t^{0,x;v}, (C_t^iX_t^{0,x;v}+D_t^iv)\rangle +\langle L_t^iX_t^{0,x;v}, X_t^{0,x;v}\rangle\right|^2\, dt\cr
  = &\ |\langle M X_T^{0,x;v}, X_T^{0,x;v}\rangle|^2-|V(t, X_t^{0,x;v})|^2\cr
  &+2\int_0^TV(t, X_t^{0,x;v})\biggl [\langle k_t X_t^{0,x;v}, X_t^{0,x;v}\rangle-F(t,X_t^{0,x;v}, v; K_t, L_t)\biggr]\, dt\cr
  &-\int_0^TV(t, X_t^{0,x;v})\biggl[2\langle K_t (C_t^iX_t^{0,x;v}+D_t^iv), X_t^{0,x;v}\rangle-\langle L_t^iX_t^{0,x;v}, X_t^{0,x;v}\rangle\biggr] \, dW_t^i. \cr}
\end{equation}
Since $V(t, X_t^{0,x;v})\ge 0$, taking $v=0$ and using the inequality~(\ref{submartingales4}), we have
\begin{equation}\label{L-estimate2}
  \eq{ &\int_0^T\left|\langle K_t X_t^{0,x;0}, C_t^iX_t^{0,x;0}\rangle +\langle L_t^iX_t^{0,x;0}, X_t^{0,x;0}\rangle\right|^2\, dt\cr
  \le &\ |M| |X_T^{0,x;0}|^4+2\int_0^TV(t, X_t^{0,x;0})l(t,X_t^{0,x;0}, 0)\, dt\cr
  &-\int_0^TV(t, X_t^{0,x;0})\left[2\langle K_t C_t^iX_t^{0,x;0}, X_t^{0,x;0}\rangle-\langle L_t^iX_t^{0,x;0}, X_t^{0,x;0}\rangle\right] \, dW_t^i. \cr}
\end{equation}
Since $V(t, X_t^{0,x;0})=\langle K_t X_t^{0,x;0}, X_t^{0,x;0}\rangle$ and $K$ is uniformly bounded, there is a positive constant $\lambda$ such that
\begin{equation}\label{L-estimate3}
  \eq{ &\int_0^T\left|\langle L_t^iX_t^{0,x;0}, X_t^{0,x;0}\rangle\right|^2\, dt\cr
  \le &\  2\int_0^T\left|\langle K_t X_t^{0,x;0}, C_t^iX_t^{0,x;0}\rangle\right|^2\, dt+2|M| |X_T^{0,x;0}|^4\cr
  & +4\int_0^TV(t, X_t^{0,x;0})l(t,X_t^{0,x;0}, 0)\, dt\cr
  &-2\int_0^TV(t, X_t^{0,x;0})\left[2\langle K_t C_t^iX_t^{0,x;0}, X_t^{0,x;0}\rangle-\langle L_t^iX_t^{0,x;0}, X_t^{0,x;0}\rangle\right] \, dW_t^i\cr
  \le & \ \lambda \max_{t\in [0,T]}| X_t^{0,x;0}|^4-4\int_0^TV(t, X_t^{0,x;0})\langle K_t C_t^iX_t^{0,x;0}, X_t^{0,x;0}\rangle\, dW_t^i \cr
  &+2\int_0^TV(t, X_t^{0,x;0})\langle L_t^iX_t^{0,x;0}, X_t^{0,x;0}\rangle \, dW_t^i. \cr}
\end{equation}
Therefore, for $p\ge 1$, we have
\begin{equation}\label{L-estimate4}
  \eq{ &E\left(\int_0^T\left|\langle L_t^iX_t^{0,x;0}, X_t^{0,x;0}\rangle\right|^2\, dt\right)^p\cr
    \le & \ \lambda_p E\left[\max_{t\in [0,T]}| X_t^{0,x;0}|^{4p}\right]\cr
    &+\lambda_p E \left|\int_0^TV(t, X_t^{0,x;0})\langle K_t C_t^iX_t^{0,x;0}, X_t^{0,x;0}\rangle\, dW_t^i\right|^p\cr
  &+\lambda_p E\left|\int_0^TV(t, X_t^{0,x;0})\langle L_t^iX_t^{0,x;0}, X_t^{0,x;0}\rangle \, dW_t^i\right|^p \cr
  \le & \ \lambda_p E\left[\max_{t\in [0,T]}| X_t^{0,x;0}|^{4p}\right]\cr
    &+\lambda_p E \left[\int_0^T\left|V(t, X_t^{0,x;0})\langle K_t C_t^iX_t^{0,x;0}, X_t^{0,x;0}\rangle\right|^2\, dt\right]^{p/2}\cr
  &+\lambda_p E\left[\int_0^T\left|V(t, X_t^{0,x;0})\langle L_t^iX_t^{0,x;0}, X_t^{0,x;0}\rangle \right|^2\, dt\right]^{p/2} \cr
   \le & \ \lambda_p E\left[\max_{t\in [0,T]}| X_t^{0,x;0}|^{4p}\right]+\lambda_p E\left[\int_0^T|\langle L_t^iX_t^{0,x;0}, X_t^{0,x;0}\rangle|^2|X_t^{0,x;0}|^4\, dt\right]^{p/2} \cr
    \le & \ \lambda_p E\left[\max_{t\in [0,T]}| X_t^{0,x;0}|^{4p}\right]\cr
    &\quad\quad+\lambda_p E\left[\left(\int_0^T|\langle L_t^iX_t^{0,x;0}, X_t^{0,x;0}\rangle|^2\, dt\right)^{p/2}\max_{t\in [0,T]} |X_t^{0,x;0}|^{2p}\right]\cr
    \le & \ \lambda_p E\left[\max_{t\in [0,T]}| X_t^{0,x;0}|^{4p}\right]+{1\over2} E\left[\left(\int_0^T|\langle L_t^iX_t^{0,x;0}, X_t^{0,x;0}\rangle|^2\, dt\right)^p\right].  \cr}
\end{equation}
Consequently, we have for any $x\in \mathbb{R}^n$,
\begin{equation}\label{L-estimate5}
  E\left(\int_0^T\left|\langle L_t^iX_t^{0,x;0}, X_t^{0,x;0}\rangle\right|^2\, dt\right)^p \le  \ 2\lambda_p E\left[\max_{t\in [0,T]}| X_t^{0,x;0}|^{4p}\right]\le \lambda_p' |x|^{4p},
\end{equation}
which implies the following inequality
\begin{equation}\label{L-estimate6}
  E\left(\int_0^T\left|\Phi_t' L_t^i\Phi_t\right|^2\, dt\right)^p \le  \lambda_p.
\end{equation}
Hence,
\begin{equation}\label{L-estimate7}
\eq{ &  E\left(\int_0^T\left|L_t^i\right|^2\, dt\right)^p
\le  \ E\left(\int_0^T\left|\Psi_t'\Phi_t' L_t^i\Phi_t\Psi_t\right|^2\, dt\right)^p \cr
\le  & \ E\left(\int_0^T|\Psi_t'|^2|\Psi_t|^2 \left|\Phi_t' L_t^i\Phi_t\right|^2\, dt\right)^p \cr
\le &  \ E\left[\left(\int_0^T \left|\Phi_t' L_t^i\Phi_t\right|^2\, dt\right)^p \max_{t\in [0,T]}|\Psi_t|^{4p}\right]\cr
\le  & \ \left\{E\left[\left(\int_0^T \left|\Phi_t' L_t^i\Phi_t\right|^2\, dt\right)^{2p}\right]E\left[\max_{t\in [0,T]}|\Psi_t|^{8p}\right]\right\}^{1/2}\le \
\lambda_p. \cr}
\end{equation}

The proof is complete.
\end{proof}

\begin{remark} We have shown in Steps 1 and 2 that $(K,L)$ solves BSRE~(\ref{bsre}) with $K$ being nonnegative and uniformly bounded. Then from Tang~\cite[Theorem 5.1, page 62]{Tang2003}, we have the desired estimate. Here we have given a different proof to the estimate~(\ref{p-estimate}).
\end{remark}

Immediately, we have the following existence of adapted solution to BSRE~(\ref{bsre}).

\begin{corollary} (Existence result for BSRE). Let assumptions (A1) and (A2) be satisfied. Then  $(K,L)$ is an adapted solution to BSDE~(\ref{bsre}).
\end{corollary}

\section{Verification theorem and uniqueness result for BSRE}
In the theory of linear quadratic optimal stochastic control, the Riccati equation as a nonlinear system of backward (stochastic) differential equations is an equivalent form of the underlying Bellman equation as a nonlinear backward (stochastic) partial differential equations, and both the optimal control and the value function are expected to be given in terms of the solution to the Riccati equation. The following verification theorem illustrates such a philosophy, which, however,  has more or less been addressed in the author's work~\cite[Theorem 3.2, page 60]{Tang2003}.

\begin{theorem} (Verification Theorem). Let assumptions (A1) and (A2) be satisfied. Let $(K,L)$ be an adapted solution to BSDE~(\ref{bsre}) such that $K$ is  essentially bounded and nonnegative (and consequently $L$ satisfies estimate~(\ref{p-estimate}) in view of Tang~\cite[Theorem 5.1, page 62]{Tang2003}). Then, (i) the following linear SDE
\begin{equation}\label{closed system}
\left\{\eq{ d\overline{X}_t=&\left[A_t-B_t\mathscr{N}_t^{-1}(K_t)\mathscr{M}_t'(K_t, L_t)\right]\overline{X}_t\, dt\cr
&+ \sum_{i=1}^d \left[C_t^i-D_t^i\mathscr{N}_t^{-1}(K_t)\mathscr{M}_t'(K_t, L_t)\right]\overline{X}_t\, dW_t^i, \quad t\in [0,T]; \cr
 \overline{X}_0=&x \cr}\right.
\end{equation}
has a unique strong solution $\overline{X}$ such that
\begin{equation}\label{square integrable}
E\left[\max_{t\in [0,T]}|\overline{X}_t|^2\right]< \infty;
\end{equation}
(ii) the following given process
\begin{equation}\label{control}
    \overline{u}_t=-\mathscr{N}_t^{-1}(K_t)\mathscr{M}_t'(K_t, L_t)\overline{X}_t, \quad t\in [0,T],
\end{equation}
belongs to $\mathscr{L}^2_{\mathscr{F}}(0,T;\mathbb{R}^m)$, and is the optimal control for the SLQ;
and (iii) the value field $V$ is given by
\begin{equation}\label{value vield}
    V(t,x)=\langle K_tx, x\rangle, \quad (t,x) \in [0,T]\times \mathbb{R}^n.
\end{equation}
\end{theorem}

\begin{remark} A proof using the stochastic maximum principle (the so-called stochastic Hamilton system) is given in Tang~\cite[Section 3, pages 58--60]{Tang2003}. The main difficulty of the proof comes from the appearance of $L$ in the optimal feedback law~(\ref{control}) since $L$ is in general not expected to be essentially bounded. Since the coefficients of the optimal closed system~(\ref{closed system}) contain $L$, we could directly have neither the integrability~(\ref{square integrable}) nor the square integrability of $\overline{u}$, which prevent us from going through the conventional method of ``completion of squares" in a straightforward way. In what follows, we get around the difficulty via the technique of localization by stopping times, and develop a localized version of the conventional method of ``completion of squares", which give a different self-contained proof.
\end{remark}

\begin{proof} Since the coefficients of the optimal closed system~(\ref{closed system}) is square integrable on $[0,T]$ almost surely, SDE~(\ref{closed system}) has a unique strong solution ${\overline X}$ (see Gal'chuk~\cite{Galchuk1978}).  Define for sufficiently large integer $j$, the stopping time $\tau_j$ as follows:
\begin{equation}\label{localization}
    \tau_j:=T\wedge \min \{t\ge 0: \ |{\overline X}_t|\ge j\},
\end{equation}
with the convention that $\min \emptyset=\infty$. It is obvious that $\tau_j\uparrow T$ almost surely as $j\uparrow \infty$.
Then, we have
\begin{equation}\label{identity}
    \langle K_0x,x \rangle= E\langle K_{\tau_j}{\overline X}_{\tau_j}, {\overline X}_{\tau_j}\rangle +E\int_0^{\tau_j}l(t,{\overline X}_t, \overline{u}_t )\, dt,
\end{equation}
which together with assumption (A2) implies the following (with the constant $\d>0$)
\begin{equation}
   E\int_0^{\tau_j}|{\overline u}_t|^2 \, dt \le\ \d^{-1}  E\int_0^{\tau_j}\langle N_t {\overline u}_t, \overline{u}_t\rangle \, dt\le \  \d^{-1}\langle K_0x,x \rangle.
\end{equation}
Using Fatou's lemma, we have $\overline{u}\in \mathscr{L}^2_{\mathscr{F}}(0,T;\mathbb{R}^m)$. Since ${\overline X}=X^{0,x;\overline{u}}$, we have from estimate~(\ref{state-estimate}) the integrability~(\ref{square integrable}). Assertion (i) has been proved.

From Assertion (i), we see that
$$
0\le \langle K_{\tau_j}{\overline X}_{\tau_j}, {\overline X}_{\tau_j}\rangle \le \lambda\max_{t\in [0,T]} |{\overline X}_t|^2 \in L^1(\Omega, \mathscr{F}_T, P)
$$
and
$$
0\le \int_0^{\tau_j}l(t,{\overline X}_t, \overline{u}_t )\, dt \ \le \ (\hbox{ \rm and } \bigm\uparrow ) \ \int_0^T l(t,{\overline X}_t, \overline{u}_t )\, dt \in L^1(\Omega, \mathscr{F}_T, P).
$$
Using Lebesgue's dominant convergence theorem, we have
\begin{equation}\label{dominant convergence}
\eq{\lim_{j\to \infty}E\langle K_{\tau_j}{\overline X}_{\tau_j}, {\overline X}_{\tau_j}\rangle=& E\langle K_T{\overline X}_T, {\overline X}_T\rangle, \cr
\lim_{j\to \infty} E\int_0^{\tau_j}L(t,{\overline X}_t, \overline{u}_t )\, dt=& E\int_0^T l(t,{\overline X}_t, \overline{u}_t )\, dt.\cr}
\end{equation}
In view of the equality~(\ref{identity}), we have
\begin{equation}\label{identity at limit}
 \langle K_0x,x \rangle=E\langle K_T{\overline X}_T, {\overline X}_T\rangle+E\int_0^T l(t,{\overline X}_t, \overline{u}_t )\, dt=J(\overline{u};0,x).
\end{equation}
It remains to prove that for any $u\in \mathscr{L}^2_{\mathscr{F}}(0,T;\mathbb{R}^m)$, we have $J(u;0,x)\ge \langle K_0x,x \rangle$.

For given $u\in \mathscr{L}^2_{\mathscr{F}}(0,T;\mathbb{R}^m)$ and sufficiently large integer $j$, define the stopping time $\tau_j^u$ as follows:
\begin{equation}\label{localization*}
    \tau_j^u:=T\wedge \min \{t\ge 0: \ |X_t^u|\ge j\},
\end{equation}
with the notation $X^u:=X^{0,x;u}$. It is obvious that $\tau_j^u\uparrow T$ almost surely as $j\uparrow \infty$.
Define
\begin{equation}\label{control*}
    \widetilde{u}_t:=-\mathscr{N}_t^{-1}(K_t)\mathscr{M}_t'(K_t, L_t)X^u_t, \quad t\in [0,T].
\end{equation}
Then, the restriction of $\widetilde{u}$ to the random time interval $[0,  \tau_j^u]$ lies in $\mathscr{L}^2_{\mathscr{F}}(0, \tau_j^u;\mathbb{R}^m)$ for any $j$.  Using BSRE~(\ref{bsre}) to complete the square in a straightforward manner, we have
\begin{equation}\label{identity*}
  \eq{ & E\langle K_{\tau_j^u}X^u_{\tau_j^u}, X^u_{\tau_j^u}\rangle +E\int_0^{\tau_j^u}l(t, X^u_t, u_t )\, dt\cr
    &\quad\quad\quad\quad=  \ \langle K_0x,x \rangle+E\int_0^{\tau_j^u}\langle \mathscr{N}_t^{-1}(K_t) (u_t- \widetilde{u}_t),  u_t- \widetilde{u}_t\rangle\, dt. \cr}
\end{equation}
Therefore, we have
\begin{equation}\label{inequality*}
    E\langle K_{\tau_j^u}X^u_{\tau_j^u}, X^u_{\tau_j^u}\rangle +E\int_0^{\tau_j^u}l(t, X^u_t, u_t )\, dt\ge \langle K_0x,x \rangle.
\end{equation}
In view of estimate~(\ref{state-estimate}) in Lemma~\ref{sde estimate}, we see that
$$
0\le \langle K_{\tau_j^u}X^u_{\tau_j^u}, X^u_{\tau_j^u}\rangle\le \lambda\max_{t\in [0,T]} |X_t^u|^2\in L^1(\Omega, \mathscr{F}_T, P)
$$
and
$$0\le \int_0^{\tau_j^u}l(t, X^u_t, u_t )\, dt \ \le \ (\hbox{ \rm and } \bigm\uparrow ) \  \int_0^T l(t, X^u_t, u_t )\, dt\in L^1(\Omega, \mathscr{F}_T, P).
$$
Passage to the limit in inequality~(\ref{inequality*}), again using Lebesgue's dominant convergence theorem, we have
\begin{equation}
    J(u; s,x)= E\langle K_TX^u_T, X^u_T\rangle +E\int_0^Tl(t, X^u_t, u_t )\, dt\ge \langle K_0x,x \rangle.
\end{equation}
The proof is then complete.
\end{proof}

Immediately, we have the following uniqueness of adapted solution to BSRE~(\ref{bsre}).

\begin{corollary} (Uniqueness result for BSRE). Let assumptions (A1) and (A2) be satisfied. Let $(\widetilde{K}, \widetilde{L})$ be an adapted solution to BSDE~(\ref{bsre}) such that ${\widetilde K}$ is  essentially bounded and nonnegative and $\widetilde{L}$ satisfies estimate~(\ref{p-estimate}). Then, ${\widetilde K}=K$ and $\widetilde{L}=L$.
\end{corollary}

The corollary and its proof can be found in Tang~\cite[the beginning paragraph of Section 8, page 70]{Tang2003}.

\section{Comments and possible extensions}
The results of this paper can be adapted to the singular case ($N$ is allowed to be
only nonnegative) but with suitable additional conditions such as the following:

(A3) Assume that the matrix process $\sum_{i=1}^d(D^i)'D^i$
 and the terminal state weighting random matrix $M$ are uniformly positive.

This subject will be detailed elsewhere.

 The singular case has received much recent interests because of its appearance in
 financial mean-variance problems. More generally, $N$ can also be possibly
 negative---this is the so-called indefinite case. On these
 features, the interested reader  is referred to  Chen and Yong
 \cite{new5}, Hu and Zhou~\cite{HuZhou2003},
 Kohlmann and Tang \cite{8, 11}, Yong and Zhou \cite{YongZhou}, and the references therein.

 Finally, the main results of the paper can also be adapted to the quadratic optimal control problem for linear
 stochastic differential system driven by jump-diffusion processes under suitable assumptions. The details will be presented elsewhere.

Consider a general non-Markovian nonlinear optimal stochastic control problem. Let $A$ be a separable metric space,
and ${\mathscr U}_s$ be the set of $A$-valued predictable processes on $[s,T]$.

For any triplet $(u, s, \xi)\in {\mathscr U}_s\times [0,T]\times L^2(\Omega, \mathscr{F}_s, P; \mathbb{R}^n)$, consider the following SDE:
$$
X_t=\xi+\int_s^t\sigma(r,X_r, u_r)\, dW_r+\int_s^tb(r,X_r,u_r)\, dr, \quad t\in [s,T].
$$

Assume that the following functions
$$\begin{array}{c}
\sigma(t,x,\alpha)\in \mathbb{R}^{n\times d},\quad  b(t,x,\alpha)\in \mathbb{R}^n,\\
  l(t,x, \a)\in \mathbb{R}, \quad g(x)\in \mathbb{R}; \qquad (t,x,\alpha)\in [0,T]\times \mathbb{R}^n\times A
 \end{array}$$ are continuous in $(x, \alpha)$ and continuous in $x$ uniformly over $\alpha$ for each $(t, \omega)$. Also, assume thatthere is positive constant $\l$ such that
$$\begin{array}{rcl}
\|\sigma(t,x,\alpha)-\sigma(t,y,\alpha)\|+|b(t,x,\alpha)-(t,y,\alpha)|&\le& \l |x-y|,\\
\|\sigma(t,x,\alpha)\|+|b(t,x,\alpha)|&\le& \l (1+|x|),\\
|l(t,x, \a)|+|g(x)|&\le& \l (1+|x|)^m.
\end{array}$$

For $(s, \xi, u) \in [s,T]\times L^2(\Omega, \mathscr{F}_s, P; \mathbb{R}^n)\times {\mathscr U}_s,$ define
$$\begin{array}{rcl}
J(u;s,\xi)&=&\displaystyle E^{s,\xi;u}\left[\int_s^Tl(t,X_t, u_t)\, dt+g(X_T)\biggm | {\mathscr F}_s \right],\\[4mm]
V(s,\xi)&:=&\displaystyle\hbox{\rm ess.}
\inf_{u\in {\mathscr U}_s} J(u; s, \xi).
\end{array}$$
Denote by $\mathbb{V}(s,\cdot)$ the restriction of $V(s,\cdot)$ to $\mathbb{R}^n$. In the nonlinear context, the restricted value field $\mathbb{V}$ can be proved to satisfy  the stochastic dynamic programming principle:

 (i) For $s\le t\le T$ and $\xi\in L^2(\Omega, \mathscr{F}_s, P; \mathbb{R}^n)$,
$$
\mathbb{V}(s,\xi)=\mbox{\rm ess.}\inf_{u\in {\mathscr U}_s} E^{s,\xi; u} \left \{\int_s^t l(r, X_r, u_r)\, dr+ \mathbb{V}(t, X_t)  \biggm | {\mathscr F}_s\right\}.
$$

(ii) For $(s,x,u)\in [0,T]\times R^n\times {\mathscr U}_s$, the process
$$\begin{array}{rcl}
 \k_t^{s,x;u}&:=&\displaystyle  \mathbb{V}(t,X_t^{s,x;u}) +\int_s^tl(r, X_r^{s,x;u}, u_r)\, dr
\end{array}
$$
defined for $t\in [s,T]$, is a submartingale w.r.t. $\{{\mathscr F}_t\}$.

Using the above dynamic programming principle and Kunita's stochastic calculus~\cite{Kunita1994}, we can still show that $\mathbb{V}$ is a Sobolev space valued semi-martingale and satisfy the associated backward Bellman equation in the strong sense. All the details shall be given in our forthcoming paper to extend Krylov~\cite{Krylov1972} to the non-Markovian framework for optimal stochastic control problem.

\section*{Acknowledgment} The main results and the methodology of the paper has been announced in my plenary talk at the 7th international symposium on backward stochastic differential equations (June 22-27, 2014), Weihai, Shandong Provence, China. The author would thank the organizers for kind hospitality.

\rm

\end{document}